\documentclass{amsart}

\usepackage{url}

\newtheorem{theorem}{Theorem}[section]

\newtheorem{lemma}[theorem]{Lemma}

\theoremstyle{definition}
\newtheorem{definition}[theorem]{Definition}

\theoremstyle{remark}
\newtheorem{remark}[theorem]{Remark}
\newtheorem{example}[theorem]{Example}

\numberwithin{equation}{section}

\def\Ind{\setbox0=\hbox{$x$}\kern\wd0\hbox to 0pt{\hss$\mid$\hss} \lower.9\ht0\hbox to 0pt{\hss$\smile$\hss}\kern\wd0} 

\def\Notind{\setbox0=\hbox{$x$}\kern\wd0\hbox to 0pt{\mathchardef \nn=12854\hss$\nn$\kern1.4\wd0\hss}\hbox to 0pt{\hss$\mid$\hss}\lower.9\ht0 \hbox to 0pt{\hss$\smile$\hss}\kern\wd0}

\def \d {\delta}
\def \DD {\mathcal D}
\def \D {\Delta}

\def \U {\mathcal U}
\def \dcl {dcl}

\def \l {\langle}
\def \r {\rangle}
\def \G {G}
\def \s {\sigma}

\def \Aut {Aut}
\def \DCF {\operatorname{DCF}}

\title[Some definable Galois theory and  examples]{Some definable Galois theory and examples}
\author{Omar Le\'on S\'anchez}
\address{Omar Le\'on S\'anchez\\
McMaster University\\
Department of Mathematics and Statistics\\
1280 Main St W\\
Hamilton, ON, L8S 4L8.}
\email{oleonsan@math.mcmaster.ca}

\author{Anand Pillay}
\address{Anand Pillay\\
University of Notre Dame\\
Department of Mathematics\\
255 Hurley, Notre Dame\\
IN, 46556.}
\email{apillay@nd.edu}
\thanks{Anand Pillay was supported by NSF grant DMS-1360702}

\date{\today}

\begin{document}

\begin{abstract}
We  make explicit certain results around the Galois correspondence in the context of definable automorphism groups, and  point out the relation to some  recent papers dealing with the Galois theory of algebraic differential equations when the constants are not ``closed" in suitable senses. We also improve the definitions and  results from \cite{DGT1} on generalized strongly normal extensions, using this to give a restatement of a conjecture from \cite{BP}. 
\end{abstract}

\maketitle

\section{Introduction}
In this partly expository paper, we  make explicit and elaborate on, results  which are essentially contained in the literature: 
 see \cite[Section B.4]{Hrushovski} by Hrushovski or  \cite[Section 2.2]{Kamensky} by Kamensky.   This is  parts (iii),  (iv), and (v)  of Theorem 2.2 below, a certain ``indirect" Galois correspondence in the context of definable automorphism groups in first order theories. 

Several papers and/or preprints have appeared recently in which a Galois correspondence is  given in various differential algebraic contexts with the common feature that the ``constants" are not assumed to be  algebraically, differentially, or in some other suitable sense, ``closed".  We simply want to clarify how
 these results are subsumed by the above references. In fact in the special case of strongly normal extensions of a differential field $K$ whose field of constants $C_{K}$ is not necessarily algebraically closed, the results are already in Kolchin's book \cite[Chapter VI, Theorem 3]{Kolchin}.  In all these cases some translation between different languages is required, and  the model-theoretic references may be unknown or rather  obscure, at least for differential algebraists.  

The reason we talk about indirect Galois correspondences is as follows: In usual Galois theory, we take for $K$ say a perfect field, $P(x) = 0$ a polynomial equation over $K$, and $L$ a splitting field for this equation.  We can consider the (finite) set $Y$ of solutions of the equation in the algebraic closure $K^{alg}$ of $K$. Then $L$ is generated over $K$ by $Y$, and moreover the (finite) group of permutations of $Y$ which preserve all polynomial relations over $K$ coincides with $Aut(L/K)$. The Galois correspondence is directly  between subgroups of $Aut(L/K)$ and fields in between $K$ and $L$.  This picture also holds for  the Picard-Vessiot theory when the constants are algebraically closed. Namely we take a linear differential equation $\d y = Ay$ over a differential field $(K,\d )$ of characteristic zero with algebraically closed field $C_{K}$ of constants. We now  let $Y$ be the solution set of the equation in the differential closure $K^{diff}$ of $K$,  which is an $n$-dimensional vector space over the field $C$ of constants of $K^{diff}$. We choose a basis $\bar y$ of $Y$ over  $C$, and let $L$ be the (differential)  field generated over $K$ by $\bar y$, the so-called  Picard-Vessiot extension of $K$ for the equation $\d y = Ay$.  It follows from $C_{K}$ being algebraically closed that actually $C = C_{K}$. Hence the full solution set $Y$ is contained in $L$, and the group of  permutations of $Y$ preserving all differential polynomial relations over $K$ and $C$, which has the the structure of an algebraic subgroup $G$ of $GL_{n}(C)$, coincides with $Aut_{\d}(L/K)$  The Galois correspondence is between algebraic subgroups of $G = Aut_{\d}(L/K)$ and differential fields in between $K$ and $L$, as in the polynomial case.  But let us consider now the differential situation when the field of constants $C_{K}$ of $K$ is {\em not necessarily} algebraically closed.  Then $C$, the algebraic closure of $C_{K}$, may properly contain $C_{K}$. The set $Y$ of solutions of the equation in $K^{diff}$ is still an $n$-dimensional vector space over $C$,  and the group $G$ of permutations of $Y$ which preserve differential polynomial relations over $K$ and $C$ is an algebraic subgroup of $GL_{n}(C)$ as before. It may happen that there is a basis $\bar y$ of $Y$ over $C$ such that the differential field $L = F(\bar y)$ has the same constants as $K$, in which case $L$ is called a Picard-Vessiot extension of $K$ for the equation.  The existence of such $L$ implies that $G$ {\em is} defined over $C_{K}$, and we have moreover that $G(C)$ coincides with $Aut_{\d}(L(C)/K(C))$ while $G(C_{K})$ coincides with $Aut_{\d}(L/K)$. There is in general {\em no} direct Galois correspondence between algebraic subgroups of $G(C_{K})$ and  differential fields in between $K$ and $L$. But there {\em is} rather a Galois correspondence between algebraic subgroups of $G(C)$, defined over $C_{K}$,  and differential fields in between $K$ and $L$: given such an intermediate differential field $F$, the elements of $G$ which fix $F$ pointwise form an algebraic subgroup of $G$ defined over $C_{K}$, etc.   This is an example of the {\em indirect} Galois correspondence,  which is given by the existence of a Picard-Vessiot extension, and which is a special case of the model theoretic result we will be presenting. 

The general model-theoretic context is roughly (working inside an ambient saturated structure $\bar M$): we are given a set $A$ of parameters, and $A$-definable sets $Y$ and $X$. Under suitable conditions $Aut(Y/A,X)$, the group of permutations of $Y$ which are elementary over $A\cup X$ will have the structure of an $A$-definable group $G$.  And under additional conditions we will study the ``extension"  $B/A$ where $B = dcl(A,b)$ for a suitable $b\in Y$, as well as $Aut(B/A)$, which are the general versions of a Picard-Vessiot extension  $L/K$, and its automorphism group $Aut(L/K)$.  And we will describe the relations between these objects as well as the (indirect) Galois correspondence.
This will be accomplished in  the next section,  where the basic set-up (or definition) is given by conditions (I) and (II) and the main result is Theorem 2.3.  
As remarked in \cite{Hrushovski} what is being discussed is when and how the theory of {\em definable} automorphism groups translates to a theory of extensions $B/A$ of definably closed sets and their automorphism groups.

In the final section we show how this set-up and Theorem  2.3 subsume the various differential algebraic examples.  One possibly novel  thing on the differential algebraic side will be a refinement of the generalized strongly normal theory from \cite{DGT1}, \cite{DGT4},  where we are able to  drop the assumption that $X(K) = X(K^{diff})$.  We will use this to  give a rather clearer statement of Conjecture 2.3  from \cite{BP},  which is related to Ax-Lindemann-Weierstrass questions for (families of) semiabelian varieties.  This is possibly the only part of the paper which requires some more background from the reader.

We should make it clear that  our aim is {\em not} to give yet another exposition of  definable automorphism groups and/or  the model theoretic approach to the Galois theory of differential equations,  of which there are many  (\cite[Section 8.3]{Poizat-book}, \cite{Pillay-LMS}, \cite[[Chapter 8, section 4]{Pillay-book}), but rather to explain concisely how a number of recent results in the literature follow from general and known model-theoretic considerations.

\section{Model theoretic context and results}

Throughout this section we assume familiarity with the basics of model theory (the material in \cite[Chapter 1]{MMP} covers most of what we use here). 
Let us fix a complete theory $T$ in the language $L$.  $\bar M$ denotes a saturated model of $T$, and A, B, ... small subsets of $\bar M$. 
For convenience we assume 
\begin{center}
($\dagger$) \quad $T$ is $\omega$-stable and $T = T^{eq}$
\end{center}

In place of the $T = T^{eq}$ assumption, the reader could just assume that $T$ is $1$-sorted with elimination of imaginaries, which will actually be the case in the differential algebraic contexts and  examples, where $T = \DCF_{0,m}$.  

We fix a definably closed set $A$, a complete type $q(x)$ over $A$, and an $A$-definable set $X$. We consider the following conditions on $q$ and $X$.
\begin{enumerate}
\item [(I)] For any realizations $b,b'$ of $q$, $b'\in dcl(b,X,A)$.
\item   [(II)]  For any tuple $c$ from $X$, $q(x)$ and $r(y) = tp(c/A)$ are weakly orthogonal, namely $q(x)\cup r(y)$ extends to a unique complete $xy$-type over  $A$. 
\end{enumerate}  

\vspace{5mm}
\noindent

We do not really wish to add new and  unnecessary terminology to the subject, but we mention some possibilities in the following remark. 
\begin{remark} \
\begin{enumerate}
\item Condition (I) says that $q$ is internal to $X$ and is already the type of a ``fundamental system of solutions".   We could use the expression ``$q$ is strongly internal to $X$" for (I).
\item Condition (II) says that for any two realizations $b_{1}, b_{2}$ of $q$ in ${\bar M}$, 
$$tp(b_{1}/A,X) = tp (b_{2}/A,X),$$ which has to be the unique nonforking extension of $q$ over $A\cup X$. 
\item   We might also want to express conditions (I) and (II) by ``$B$ is a Galois extension of $A$, relative to $X$"  where $B = dcl(A,b)$ and $b$ realizes $q$. 
\end{enumerate}

\end{remark} 

\begin{lemma}\label{isolated} \
\begin{enumerate}
\item [(i)] Let $b$ realize 
 $q$ and let $B = dcl(A,b)$. Then $q$ and $X$ satisfy condition $\operatorname{(II)}$ above if and only if
\begin{itemize}
\item [(II')]$dcl((A\cup X)\cap B) = A$.
\end{itemize}
\item [(ii)] Suppose $\operatorname{(I)}$ and $\operatorname{(II)}$ hold  (for $q,X$).  Then $q$ is isolated. 
\end{enumerate}

\end{lemma}
\begin{proof} (i) (II) implies (II') is clear. For the converse: By $\omega$-stability there is $c\in dcl(A,X)$ such that $tp(b/A,X)$ is definable over $c$ and note that $c\in dcl(b,A)$.  So $c\in dcl((A\cup X)\cap B)$. By (II'), $c\in A$, which means that $tp(b/A,X)$ is definable over $A$.  Note that for any other realization $b'$ of $q$ there is an $A$-automorphism of $\bar M$ taking $b$ to $b'$ and fixing $X$ setwise, hence $tp(b'/A\cup X)$ is definable over $A$ by the same schema defining $tp(b/A\cup X)$, and hence $tp(b'/A\cup X) = tp(b/A\cup X)$, giving (II). 
\newline
(ii) By (I) and compactness there is a formula $\phi(x)\in q(x)$ and a partial $A$-definable function $f(-,-)$ such that for all $b,b'$ satisfying $\phi(x)$, there is a tuple $c$ of elements of $X$ such that $f(b,c) = b'$. Now, by $\omega$-stability, let $M_{0}$ be a prime model over $A\cup X$, and let $b$ realize $\phi(x)$ in $M_{0}$ and let $b'$ be any realization of $q$  in ${\bar M}$. By what we just said there is $c$ from $X$ such that $f(b,c) = b'$. As $tp(b/A,X)$ is isolated, so is $tp(b'/A,X)$.  But by (II), $q$ has a unique extension to a complete type over $A\cup X$, which must be $tp(b'/A, X)$. By compactness $q$ is isolated. 
\end{proof}

We are still fixing $A$, $q$ and $X$. 
Let $Q$ be the set of realizations of $q$ (in $\bar M$). Let $Aut(Q/X,A)$ denote the group of permutations of $Q$ which are elementary over $A\cup X$, namely preserve all relations (on arbitrary $X^{n}$) which are $A\cup X$-definable.  Equivalently (using stability), $Aut(Q/X,A)$ is the group of permutations of $X$ induced by automorphisms of $\bar M$ which fix pointwise $A\cup X$.  Let us fix a realization $b$ of $q$, and let $B = dcl(A,b)$. Let $Aut(B/A)$ denote the group of permutations of $B$ which are elementary over $A$ in the sense of ${\bar M}$.  With this notation we have:

\begin{theorem}\label{Galoistheory}
Suppose that $q$ and $X$ satisfy conditions $\operatorname{(I)}$ and $\operatorname{(II)}$.  Then
\begin{enumerate}
\item [(i)] $Aut(Q/A,X)$ acts regularly (in particular transitively) on $Q$. Indeed every $\sigma\in Aut(Q/A,X)$ is determined by $\sigma(b')$ for some/any $b'\in Q$ and all elements of $Q$ have the same type over $A\cup X$. 
 \item [(ii)] There is an $A$-definable group $G$ whose domain is an $A$-definable set contained in $\dcl(A,X)$, and a group isomorphism $\mu$ between $Aut(Q/A,X)$ and $G$ such that the induced action of $G$ on $Q$ is $B$-definable.
\item[(iii)] $\mu$ induces an isomorphism between $Aut(B/A)$ and $G(A)$, the group of elements of $G$ which are in $A$.
 \item [(iv)] There is a Galois correspondence between definably closed sets in between $A$ and $B$ and $A$-definable subgroups of $G$. More precisely, for a definably closed set $C$ with $A\subseteq C\subseteq B$, if we let 
 $$H_C=\{\s \in G: \s \text{ fixes } C \text{ pointwise}\},$$
 then $H_C$ is an $A$-definable subgroup of $G$, every $A$-definable subgroup of $G$ appears as some $H_C$, and $C\subseteq C'$ iff $H_{C'}\leq H_{C}$. In addition:
 \item [(v)] Let $C$ be a definably closed set with $A\subseteq C\subseteq B$. Then, then $tp(b/C)$  and $X$ also satisfy conditions $\operatorname{(I)}$ and $\operatorname{(II)}$, and $H_C$ is isomorphic to $\Aut(Q_C/C,X)$, where $Q_C$ is the set of realizations of $tp(b/C)$. Moreover, there is a finite tuple $e$ such that $C=\dcl(e,A)$, and the type $tp(e/A)$ satisfies $\operatorname{(I)}$ if and only if $H_C$ is a normal subgroup of $G$. In this case, $G/H_C$ is isomorphic to $\Aut(P/A,X)$ where $P$ is the set of realizations of $tp(e/A)$.
\end{enumerate}
\end{theorem}
\begin{proof}
Everything is explicit or implicit in the literature, but we give a quick proof of the theorem in its entirety. (i) and (ii) set up the objects, and the main points are (iii), (iv), and (v). 

The proof of (i) is contained in the statement.

\vspace{2mm}
\noindent
(ii) In Lemma \ref{isolated} we saw that 
$$\text{(*)}\quad Q \text{ is an $A$-definable set,}  $$
and moreover by (II) (and the proof of Lemma 2.1 (i))
\begin{center}
(**) \quad $q$ has a unique extension to a complete type over  $A\cup X$ \\
which is moreover definable over $A$.
\end{center}
By (I) and compactness, there is an $A$-definable (partial) function $f_0$ such that every element of $Q$ is of the form $f_0(b,c)$ for some tuple $c$ (of fixed length) from $X$. Let $Y_0$ be the set of tuples of this fixed length  from $X$ such that $f_0(b,c)\in Q$. By (*) and (**) $Y_0$ is an $A$-definable set of tuples from $X$. Let $E$ be the equivalence relation on $Y_0$ defined by $E(c_1,c_2)$ iff $f_0(b,c_1)=f_0(b,c_2)$. Then, again by (**), $E$ is $A$-definable. Hence the  set $Y_{0}/E$ which is $A$-definable in ${\bar M}^{eq}$  is contained in $dcl(A,X)$. Let $Y$ denote $Y_{0}/E$.  We can clearly rewrite $f_{0}$ as an $A$-definable map from $Q\times Y$ to $Q$ and now note that for any $b'\in Q$ there is
a \emph{unique} $d\in Y$ such that $b'=f(b,d)$.

For $\s\in \Aut(Q/A,X)$ let $\mu(\s)$ be the unique element of $Y$ such that $\s(b)=f(b,\mu(\s))$. The map $\mu$ sets up a bijection between $\Aut(Q/A,X)$ and $Y$. Using (**) one sees that the group operation on $Y$ induced by $\mu$ is $A$-definable, and the induced action on $Q$ is $B$-definable. We let $G$ denote the set $Y$ equipped with this group structure. So $G$ is an $A$-definable group and its action  on $Q$ is definable over $B$. We will often identify $G$ with $\Aut(Q/A,X)$.

\vspace{2mm}\noindent
(iii) As mentioned above $G$ lives on the $A$-definable set $Y$, so $G(A)$ is as a set the collection of elements of $Y$ which are in $dcl(A) = A$.  The nature of our identification of $G$ with $Aut(Q/X,A)$ yields that $G(A)$ consists of those $\s\in Aut(Q/X,A)$ such that $\s(b)\in B$. Here is the explanation:  $\s\in G(A)$ means $\mu(\s)\in A$ which implies $\s(b) = f(b,\mu(\s))\in B$.  Conversely, 
if $\s(b)\in B$ then  as $\mu(\s)$ is determined by $b$ and $\s(b)$, also $\mu(\s)\in B$. But $\mu(\s)\in Y\subset dcl(A,X)$, hence by Lemma 2.2 and condition (II), $\mu(\s)\in A$, as required. Note finally that we can view $Aut(B/A)$ as precisely the subgroup of $Aut(Q/X,A)$ consisting of those $\sigma$ such that $\sigma(b)\in B$. This completes the proof. 

One should note that as by (II)', $dcl(A,X) \cap B = dcl(A,X)\cap A$, $G(B)$ coincides with $G(A)$. 

\vspace{2mm}\noindent
(iv)  For the Galois correspondence the main point is that condition (II) yields that
$$\text{(***)}\quad \text{for }D_1,D_2\subseteq B, \quad D_2\subseteq\dcl(D_1,A) \text{ iff } D_2\subseteq \dcl(D_1,A,X).$$
Note that the latter is equivalent to ``$D_2$ is fixed by every $\s\in G$ which fixes $D_1$''. Moreover, for $d\in B$
$$\text{(****)} \quad \{\s\in G:\s(d)=d\} \text{ is an $A$-definable subgroup of } \G.$$
This is because if $d=h(b)$ for $h$ an $A$-definable function, then $\s\in G$ fixes $d$ iff $h(\s(b))=h(b)$ iff $h(f(b,\s))=h(b)$, which is, by (**), an $A$-definable condition on $\s$. Now let $C$ be an arbitrary subset of $B$ and $H_C$ the set of $\s\in G$ which fix $C$ pointwise. By (****) and $\omega$-stability, $H_C$ is an $A$-definable subgroup of $G$ and, by (***), the set of elements of $B$ which are fixed by $H_C$ is precisely $\dcl(A,C)$.

On the other hand, suppose $H$ is an arbitrary subgroup of $G$ and let $C$ be the set of elements of $B$ fixed pointwise by every element of $H$. Then, as the elements of $H$ also fix $A$-pointwise, it follows that $C$ is a definably closed subset of $B$ containing $A$. Moreover, by the above, $H_C$ is an $A$-definable subgroup of $G$ containing $H$.

\vspace{2mm}
\noindent (v) Let $q_{C}=tp(b/C)$, and let $Q_{C}$ denote  its set of realizations. As $q_{C}$ extends $q$, (I) holds for $q_{C}$ and $X$. As $B = dcl(A,b)$, $tp(B/A)$ has a unique extension over $A\cup X$, so $tp(bC/A)$ has a unique extension over $A\cup X$, whereby $tp(b/C)$ has a unique extension over $C\cup X$. So $q_{C}$ and $X$ satisfy (II). 

As remarked earlier any $\sigma\in Aut(Q/A\cup X)$ is determined by the value $\sigma(b)$. Hence $\{\sigma\in Aut(Q/A\cup X): tp(\s(b)/C) = tp(b/C)\}$ is a subgroup of $Aut(Q/A\cup X)$ which is naturally isomorphic to $Aut(Q_{C}/A\cup X)$.  Moreover the image of this subgroup under $\mu$ is clearly $H_{C}$. 

As $C$ is a definably closed subset of $dcl(A,b)$, $\omega$-stability implies that $C = dcl(A,e)$ for some finite tuple $e\in C$.  Note that $p = tp(e/A)$ (and $X$) automatically satisfy condition (II). 
Suppose that $tp(e/A)$ and $X$ satisfy (I). Let $\sigma\in Aut(Q_{C}/A\cup X)$ and $\tau\in Aut(Q/A\cup X)$. Let $e'= \tau(e)$. So by (I), $e'\in dcl(A,e,X)$, whereby $\sigma(e') = e'$. Hence
$\sigma^{-1}\tau\sigma (e) = e$, so $\sigma^{-1}\tau\sigma$  is in $Aut(Q_{C}/A\cup X)$.  We have shown that $Aut(Q_{C}/A\cup X)$ is normal in $Aut(Q/A\cup X)$, whence $H_{C}$ is normal in $G$.  Conversely, if $Aut(Q_{C}/A\cup X)$ is normal in $Aut(Q/A\cup X)$, then for any 
 $\rho\in H_C$ and $\s\in G$ we have $\rho(\s(e))=\s(e)$. Therefore, all the elements of $H_C$ fix $P$ pointwise (the latter is the set of realizations of $p$). Thus, $P\subseteq \dcl(e,A,X)$. This shows that $p$ and $X$ satisfy (I).  Finally, the restriction to $P$ yields a group homomorphism from $G$ to $\Aut(P/A,X)$ with kernel $H_C$. 
\end{proof}

We end this section with a couple of remarks.   The first concerns the definability properties of $G(A)$. $G$ is of course a  group definable in the ambient structure ${\bar M}$, with parameters from $A$.  What about $G(A)$ which  corresponds to $Aut(B/A)$?  We maintain the earlier assumptions, notation and conditions (I), (II), for $q$, $X$. 

\begin{remark} \
\begin{enumerate} 
\item  Let $M_{0}$ be the prime model over $A$ (which exists by $\omega$-stability). Suppose that $X(A) = X(M_{0})$, namely all points of $X$ in the model $M_{0}$ are already in $A$. Then $G(A)$ coincides with the interpretation $G(M_{0})$ of the formula over $A$ defining $G$ in ${\bar M}$,  in the elementary substructure $M_{0}$.  So bearing in mind part (iii) of Theorem 2.3, we see in this case that $Aut(B/A)$ {\em is} a definable group in the model $M_{0}$, and we easily obtain from Theorem 2.3 the Galois correspondence between subgroups of $G(M_{0})$ definable in $M_{0}$ and definably closed sets in between $A$ and $B$. 
\item If the theory $T$ is one-sorted (with elimination of imaginaries) {\em and}  has quantifier elimination, then $X$ will be a quantifier-free $A$-definable set in ${\bar M}$, hence as the universe of $G$ is $Y$ which will be an $A$-definable subset of some power of $X$, $G(A)$ will be a group (quantifier-free)  definable (without parameters) in the structure whose universe is $X(A)$ and whose relations are intersections with $X(A)^{n}$ of $A$-definable subsets of $X^{n}$.   See  Example~\ref{example} in Section~\ref{DGT} for the failure of the direct Galois correspondence in this situation. 

\end{enumerate}
\end{remark}

In the next remark we point out an extension of part (iii) of Theorem~\ref{Galoistheory}, whose proof is left to the reader.  
\begin{remark} \label{remarksec3}
Under the same assumptions and notation of Theorem~\ref{Galoistheory}, let $D$ be any definably closed set containing $B$, and let $B_{D}$ be $\dcl(b,A,X(D))$. Then $\mu$ induces an isomorphism between 
$\Aut(B_{D}/A,X(D))$ and $G(D)$.
\end{remark}

\section{Applications to and interpretation in  differential Galois theory}\label{DGT}

We now give the interpretation of the above model-theoretic results in the context of differential fields.  In these examples we will start with the most concrete, ending with the most general, for pedagogical reasons. 

We work in the language of differential rings with $m$ derivations 
$$\mathcal L_{\text{rings}}\cup\{\d_1,\dots,\d_m\}.$$ 
All model-theoretic and differential-algebraic terminology refers to this language. Recall that $\DCF_{0,m}$ is the model companion of the theory of fields of characteristic zero equipped with $m$ commuting derivations, and is called  the theory of differentially closed fields of characteristic zero (with $m$ commuting derivations). It is well known that this theory is $\omega$-stable and eliminates imaginaries, so fits into the assumption $(\dagger)$.  Moreover, it has quantifier elimination \cite{McGrail}. When $m=1$ (the so-called {\em ordinary} case) the theory is sometimes just denoted $\DCF_{0}$. In this ordinary case we write $\d$ for $\d_1$.  

The saturated model $\bar M$  is denoted 
$$(\U,\Pi=\{\d_1,\dots,\d_m\})\models \text{DCF}_{0,m}.$$ 
$K$, $L$, etc. denote   (small)  differential subfields $K$ of $\U$.  For any set $A$ of tuples from $\U$, the definable closure of $A$ in $\U$ coincides with the differential field generated by the coordinates of the members of $A$, and the (model-theoretic) algebraic closure of $A$ coincides with the field-theoretic algebraic closure of the differential field just mentioned.  $K\langle A \rangle$ denotes the differential field generated over $K$ by $A$, where $A$ is a subset of $\U$. 
For any subset  $\DD\subseteq \Pi$ and any differential subfield $F$ of $\U$ we let $F^\DD$ denote the subfield of $\DD$-constants of $F$; that is, 
$$F^\DD=\{a\in F: \d(a)=0 \text{ for all } \d \in \DD\}.$$ 

We let $\mathcal C$ denote $\U^{\Pi}$ the universal field of absolute constants, and also write $C_{F}$ for $F^{\Pi}$ (the absolute constants of $F$).  This is consistent with notation in the introduction.

\begin{lemma} Suppose $K\subseteq L$ are differential fields (i.e. differential subfields of $\U$).  Let $\DD$ be a nonemptyset subset of $\Pi$, and let $X = \U^{\DD}$. Then 
\begin{enumerate}
\item  $dcl(K\cup X)\cap L = K\langle L^{\DD}\rangle$.
\item $dcl(K\cup X)\cap L = K$ means precisely that $K^{\DD} = L^{\DD}$. 
\end{enumerate}
\end{lemma}
\begin{proof}
(i)  A consequence of the linear disjointness property of the constants (see \cite[Chap. II, \S1]{Kolchin}) is that for any intermediate differential field $K\leq E \leq K\langle \U^\DD\rangle$, we have that
$$K\langle E^\DD\rangle= E.$$
Applying this with $E=dcl(K\cup X)\cap L$ we get
$$ dcl(K\cup X)\cap L   =K\langle\U^\DD\rangle \cap L=K\langle K\langle\U^\DD\rangle \cap L^\DD\rangle=K\langle L^\DD\rangle.$$
\newline
(ii) follows directly from (i).

\end{proof}

\subsection{Picard-Vessiot extensions}
We focus first on the ordinary case, $m=1$. A linear differential equation over $K$ (in vector form)  is something of the form $\d y = Ay$ where $y$ is a $n\times 1$ column vector of indeterminates and $A$ is an $n\times n$ matrix with coefficients from $K$.  A {\em fundamental system of solutions} of the equation is a set $y_{1},.., y_{n}$ of solutions of the equation such that the determinant of the $n\times n$ matrix $(y_{1},..,y_{n})$ is nonzero.  Equivalently $(y_{1},..,y_{n})$ is a basis over $\mathcal C$ of the $\mathcal C$-vector space $V$ of solutions to the equation in $\U$.  By a {\em Picard-Vessiot extension of $K$ for the equation $\d y = Ay$ } is meant a differential field $L$ generated over $K$ by some fundamental system $(y_{1},..,y_{n})$ of solutions, and such that $C_{L} = C_{K}$. So we see from Lemma 3.1 that if $L$ is a Picard-Vessiot extension of $K$ for the equation then $L$ is of the form $dcl(K,b)$,  where $q = tp(b/K)$ and $\mathcal C$ satisfy conditions (I) and (II) from the previous section.  Hence Theorem 2.3 holds in this specific context. In this  case the construction or identification of the Galois group is more direct: for each $\s\in Aut(Q/A,\mathcal C)$, $\s(b) = b\mu(\s)$ for a unique matrix $\mu(\s)\in GL(n,{\mathcal C})$, and $\mu$ gives an isomorphism between $Aut(Q/A, {\mathcal C})$ and an algebraic subgroup $G$ of $GL(n,{\mathcal C})$ defined over $C_{L} = C_{K}$.  The Galois correspondence  is between algebraic subgroups of $G$ defined over $C_{K}$ and differential fields in between $K$ and $L$.   Theorem~\ref{Galoistheory} in this context is contained in \cite[Chap. VI]{Kolchin} in the slightly more general form of strongly normal extensions.
In any case what we have described here subsumes
 \cite[Theorem 4.4]{Dyckerhoff} and \cite{CHS}.

Everything extends to the partial case $m> 1$, where we now consider a set $\d_1 y = A_1y$, ...., $\d_m y = A_my$ of linear DE's over $K$ and where the matrices $A_{i}$ satisfy suitable integrability (or Frobenius) conditions. 

\subsection{Strongly normal extensions}
We work with any $m\geq 1$. According to Kolchin \cite{Kolchin} a strongly normal extension $L$ of a differential field $K$ is a differential field extension $L$ of $K$ such that $L$ is finitely generated over $K$ and for every  automorphism $\sigma$ of $\U$ over $K$, $\sigma(L)\subseteq L\l{\mathcal C}\r$ and  $\sigma$ fixes $C_{L}$ pointwise.  As Kolchin notes the condition that every automorphism of $\U$ over $K$ fixes $C_{L}$ is equivalent to requiring that $C_{L} = C_{K}$.

Let $L$ be  generated over $K$ by the finite tuple $b$, and let $q = tp(b/K)$. Thus  the condition for $L$ to be a strongly normal extension of $K$  means precisely that $q$ and $\mathcal C$ satisfy conditions (I) and (II) from the previous section (using Lemma 3.1).  

Hence Theorem~\ref{Galoistheory} applies. Using elimination of imaginaries, stability, and the fact that $\mathcal C$ is an algebraically closed field without additional structure, the Galois group $G$ as described in the proof of Theorem~\ref{Galoistheory}, is the group of ${\mathcal C}$-points of an algebraic group defined over $C_{K}$. Again the Galois correspondence is between algebraic subgroups of $G$ defined over $C_{K}$ and differential fields in between $K$ and $L$. 
Theorem 2.3 is again contained in Kolchin, \cite[Chap. VI]{Kolchin}. In fact see Theorems 3 and 4 there for the Galois correspondence.    This subsumes \cite{Brouette-Point}.  Moreover part (iii) that $Aut(L/K)$ is $G(A)$, includes,  by virtue of Remark 2.4(2), the statement in \cite[Theorem 3.15]{Brouette-Point} that $Aut(L/K)$ is interpretable in the field $C_{K}$. In fact one actually obtains definability (rather than just interpretability). Similarly, Remark~\ref{remarksec3} includes the statement in \cite[Corollary 3.21]{Brouette-Point} with definability in place of interpretability. 

\subsection{Parameterized Picard-Vessiot ($PPV$) extensions}   
The context is a differential field $K$ equipped with $m+1$  commuting derivations which we write as $\d_{x}, \d_{t_{1}},...\d_{t_{m}}$, and a linear differential equation $\d_{x}y = Ay$ over $K$.   Our universal domain $\U$ is a saturated model of $\DCF_{0, m+1}$, and the solution set $V$ of the equation is an $n$-dimensional vector space over $\U^{\d_{x}}$, the field of $\d_{x}$-constants of $\U$. A fundamental system of solutions is again a matrix $(y_{1},..,y_{n})$ of solutions, with nonzero determinant, and by a $PPV$-extension of $K$ for the equation we mean a differential field $L$ generated over $K$ by such a fundamental system of solutions, such that also $K^{\d_{x}} = L^{\d_{x}}$. 
If $(y_{1},..,y_{n})$ is such and its type over $K$ is $q$ then as before, $q$ and $\U^{\d_{x}}$ satisfy conditions (I) and (II).   The Galois group $G$ from Theorem~\ref{Galoistheory}  is a subgroup of $GL_{n}(\U^{\d_{x}})$ definable over $K^{\d_x}$ in the differentially closed field $(\U^{\d_{x}}, \d_{t_{1}},...,\d_{t_{m}})$.  The Galois correspondence (when $K^{\d_x}$ is not necessarily differentially closed) was established in Section 8.1 of \cite{GGO}, but again is contained in Theorem 2.3.

A strongly normal analogue of $PPV$ extensions was studied in \cite{LN} also establishing the Galois correspondence. 

\subsection{Landesman's strongly normal extensions}
Landesman \cite{Landesman} slightly generalizes Kolchin's definition of strongly normal extension by replacing ${\mathcal C} (= \U^{\Pi})$ by $\U^{\DD}$ for any nonempty subset $\DD$ of $\Pi$.  Formally he takes $\D$ to be $\Pi\setminus\DD$ and defines $L$ to be a $\D$-strongly normal extension of $K$ if $L$ is finitely generated over $K$, and  for any automorphism $\sigma$ of $\U$ over $K$, $\sigma(L)\subseteq L\l\U^{\DD}\r$ and $\sigma$ fixes $L^{\DD}$ pointwise. As in our discussion of strongly normal extensions above, this is equivalent to requiring that $q$, $\U^{\DD}$ satisfy conditions (I) and (II) from Section 2. In \cite[Theorem 2.34]{Landesman} Landesman proves the Galois correspondence when $K^{\DD}$ is differentially closed as a $\D$-field, and asks if the differentially closed assumption can be dropped. Our Theorem 2.3 says yes (for the indirect Galois correspondence).  Note that $PPV$ extensions are a special case of Landesman strongly normal extensions. 

\subsection{Generalised strongly normal extensions}
This is a direct transfer of the notions in Section 2 to the theory $\DCF_{0,m}$.  In the paper \cite{DGT1} the second author introduced $X$-strongly normal extensions of a differential field $K$, in the ordinary context,  where $X$ is an arbitrary $K$-definable set, also mentioning that things extend to the case of several derivations. The condition that $X(K^{diff}) = X(K)$ was included  in the original  definition, partly for consistency with the {\em usual} account of the
Picard-Vessiot (or strongly normal) theory where the field of constants of $K$ is assumed to be algebraically closed, and partly so that $Aut_{\d}(L/K)$ is a definable group in the differentially closed field $K^{diff}$.    But in fact the condition that $X(K^{diff}) = X(K)$,  can be dropped from Definition 2.1 of \cite{DGT1}, still preserving the Galois correspondence (but now indirect). We briefly  explain: 
 
Working now in $\DCF_{0,m}$, given a differential field $K$, and $K$-definable set $X$, we make:

\begin{definition} The differential field $L$ is an $X$-strongly normal extension of $K$, if $L$ is finitely generated over $K$, for each automorphism $\s$ of $\U$ over $K$, $\s (L)\subseteq L\langle X\rangle$, and $dcl(K\cup X)\cap L  =  K$. 
\end{definition} 
 This is precisely Definition 2.1 of \cite{DGT1}, adapted to several derivations, and  minus part (i).   On the other hand, this also states that $q$ and $X$ satisfy conditions (I) and (II) from Section 2., where $q$ is the type over $K$ of some generator of $L$ over $K$. 
 With this definition, Theorem 2.3 gives the Galois group, as well as the (indirect) Galois correspondence. The Galois group $G$ is a differential algebraic group, defined over $K$, and the Galois correspondence is between $K$-definable subgroups of $G$ and differential fields between $K$ and $L$.

\subsection {Logarithmic differential equations on algebraic  $\d$-groups} 
This is the only section of the paper which may require some additional background from the reader. 
In the $m=1$ case, when $K$ is algebraically closed, the  generalized strongly normal extensions of \cite{DGT1} come from ``logarithmic differential equations on algebraic $\d$-groups",  as explained in \cite{DGT4}.  This was suitably generalized by the first author to the $m>1$ case in \cite{LS}. In \cite{BP} Daniel Bertrand and the second author study a special class of such  algebraic $\d$-groups and logarithmic differential equations on them, so-called ``almost semiabelian $\d$-groups",  as part of a  term project around functional Lindemann-Weierstrass theorems for families of semiabelian varieties.  Conjecture 2.3 from \cite{BP} says roughly that the Galois group can be identified from the original equation via ``gauge transformations {\em over $K$}".  The aim here is to give a simple restatement of this property in terms of $X$-strongly nornal extensions of $K$ (from Definition 3.2) , for suitable $X$.

Here we  assume $m = 1$.  We refer the reader to \cite{DGT4} or \cite{BP} but we repeat some definitions. By an algebraic $\d$-group over $K$ we mean an algebraic group $G$ over $K$ together with a lifting of the dervation $\d$ on $K$ to a derivation $s$ on the structure sheaf of $G$. We can and will identify $s$ with a homomorphic rational section over $K$ of the shifted tangent bundle $T_{\d}(G)\to G$. The logarithmic derivative $dlog_{(G,s)}:G \to LG$ (to the Lie algebra of $G$)  is the differential algebraic map taking $x\in G$ to $\d (x)\cdot s(x)^{-1}$ computed in the algebraic group $T_{\d}(G)$.  A logarithmic differential equation on $(G,s)$ over $K$  is an equation of the form $dlog_{(G,s)}(-) = A$, where $A\in LG(K)$.  When $G = GL_{n}$ and $s$ is the $0$-section of $T(G) \to G$, such an equation is simply a linear differential equation in matrix form. 

$(G,s)^{\partial}$ denotes $\{x\in G: s(x) = \d (x)\}$, the ``kernel" of the $dlog_{(G,s)}$-map. It is a finite-dimensional differential algebraic group.  $K_{(G,s)}^{\sharp}$ denotes the differential field generated over $K$ by $G^{\partial}(K^{diff})$, the points of $G^{\partial}$ in the differential closure $K^{diff}$ of $K$.  As any two solutions $y_{1}, y_{2}\in G(K^{diff})$ of a logarithmic differential equation $dlog_{(G,s)}(-) = A$ on $(G,s)$ over $K$, differ by an element of $G^{\partial}(K^{diff})$, it follows that $tr.deg(K_{(G,s)}^{\sharp}(y)/K_{(G,s)}^{\sharp})$ is the same, for any solution of the equation in $G(K^{diff})$.  

We write the lemma below for the case where $G$ is commutative and use additive notation, although suitably written it holds in general. 

\begin{lemma} Let $(G,s)$ be a commutative connected algebraic $\d$-group over $K$ where $K$ is algebraically closed. Let $A\in LG(K)$ and let $y\in G(K^{diff})$ be a solution of the equation
$$\operatorname{(*)} \quad   dlog_{(G,s)}(-) = A.$$
Then the following are equivalent:
\begin{enumerate}
\item [(i)] $tr.deg(K_{(G,s)}^{\sharp}(y)/K_{(G,s)}^{\sharp})$ equals min$\{dim(H): H$ is a connected algebraic $\d$-subgroup of $G$ defined over $K$ and $A\in L(H) + dlog_{(G,s)}(G(K))\}$.
\item [(ii)] For some solution $y_0$ of $\operatorname{(*)}$, $K(y_0)$ is a $G^{\partial}$-strongly normal extension of $K$. 
\end{enumerate}
\end{lemma}
\begin{proof} The condition in (i) that $A\in L(H) + dlog_{(G,s)}(G(K))$ is obviously equivalent to  $y\in H + G(K) + G^{\partial}(K^{diff})$ and will be used that way. 
\smallskip

\noindent (i)$\,\Rightarrow\,$(ii):  Fix a solution $y$ of (*). Suppose that $H$ is a connected algebraic $\d$-subgroup of $G$ such that  $y\in H + G(K) + G^{\partial}(K^{diff})$ and 
$$dim(H) = tr.deg.(K_{(G,s)}^{\sharp}(y)/K_{(G,s)}^{\sharp}).$$
So there is $y_{0}\in H+G(K)$  of the form $y+ g$ for some $g\in G^{\partial}(K^{diff})$. Note that $y_{0}$ is also a solution of (*) and moreover $tr.deg(K(y_{0})/K) \leq dim(H)$.  On the other hand, $tr.deg.(K_{(G,s)}^{\sharp}(y_{0})/K_{(G,s)}^{\sharp})$ = $ tr.deg.(K_{(G,s)}^{\sharp}(y)/K_{(G,s)}^{\sharp})$ = $dim(H)$, by assumption.  The conclusion is that
$tr.deg(K(y_{0})/K) =  dim(H)$ and $y_{0}$ is independent from $K_{G}^{\sharp}$ over $K$ (in $\DCF_{0}$).   As $K_{G}^{\sharp} = K(G^{\partial}(K^{diff}))$, the latter says that $y_{0}$ is independent from $G^{\partial}(K^{diff})$ over $K$, and this says that $y_{0}$ is independent from $G^{\partial}$ over $K$. As $K$ is algebraically closed and $G^{\partial}$ is definable, this says that $tp(y_{0}/K)$ has a unique extension to a complete type over $K\cup G^{\partial}$ which is precisely condition (II).  As $y_{0}$ is a solution of (*) we have condition (I) too.

Conversely, suppose (ii) holds, witnessed by $y$.  By (II), $y$ is independent of $G^{\partial}$ over $K$. By \cite{BP}, Remark following Fact 2.2, let $H$ be a connected algebraic $\d$-subgroup 
of $G$, over $K$ such that $H^{\partial} = Aut(Y/K,G^{\partial})$. Consider the orbit $y + H^{\partial}$. It is defined over $K^{\sharp}$, and over $K(y)$ (in $\DCF_{0}$),  so also over $K$. But then $y+H$, the Zariski closure of $y+H^{\partial}$  is defined over $K$ (in $\DCF_{0}$).  By elimination of imaginaries in $\operatorname{ACF}$, let $e$ be a code of $y+H$ (in $\operatorname{ACF}$). But then $e\in K$. So $y+H$ is defined over $K$ in $\operatorname{ACF}$ so has a $K$-rational point $y_{1}\in G(K)$. So
 $y\in H + G(K)$. The minimality of $dim(H)$ is clear.
\end{proof}

Conjecture 2.3 from \cite{BP} can now be restated as:  Let $G$ be an almost semiabelian $\d$-group over $K = {\mathbb C}(t)^{alg}$, let $A\in LG(K)$, and consider the equation $dlog_{(G,s)}(-) = A$. Then the equation has a solution $y$ such that $K(y)$ is a $G^{\partial}$-strongly normal extension of $K$.

\medskip

Finally let us mention an example, even a Picard-Vessiot extension,  where the Galois correspondence is not direct.

\begin{example}\label{example}
We work in the ordinary case. Let $K=\mathbb Q$ and $L=\mathbb Q(e^t)$, with derivation $d/dt$.  Then $L$ is a Picard-Vessiot extension of $K$ for the equation $\d y = y$.  It can be seen that in this case $G=\mathbb G_m({\mathcal C})$. The intermediate differential fields are $\mathbb Q(e^{nt})$ as $n$ varies in $\mathbb N$. However, the only algebraic subgroups of $G(K)=\mathbb G_m(\mathbb Q)$ are the two obvious ones and $\{-1,1\}$.
\end{example}

\bibliographystyle{plain}

\end{document}